\documentclass[reqno]{amsart}
\usepackage[english]{babel}
\usepackage{amscd,amssymb,amsmath,amsfonts,latexsym,amsthm}
\usepackage{inputenc}
\usepackage{graphicx}
\usepackage{longtable}
\usepackage{hyperref,amsthm,amsmath,amssymb,color,multirow,graphicx}
\usepackage[fontsize=11pt]{scrextend}

\newtheorem{thm}{Theorem}[section]

\newtheorem{prop}[thm]{Proposition}
\theoremstyle{definition}
\newtheorem{defn}[thm]{Definition}

\numberwithin{equation}{section}

\newcommand{\be}{\begin{equation}}

\newcommand{\ee}{\end{equation}}

\begin{document}

\sloppy


\begin{center}
\textbf{\large
Lie affgebra structures on three-dimensional solvable Lie algebras}\\

\textbf{Berdalova Kh.R., Khudoyberdiyev A.Kh. }
\end{center}

\textbf{Abstract.}
In this work, the complex Lie affgebra structures on three-dimensional solvable Lie algebras are completely determined.

\textbf{Keywords:}
Lie algebras, Lie affgebras, solvable Lie algebras, generalized derivations.

\textbf{MSC (2020): 17B05; 17B30; 20N10.}

\makeatletter
\renewcommand{\@evenhead}{\vbox{\thepage \hfil {\it Berdalova Kh.R., Khudoyberdiyev A.Kh.}  \hrule }}
\renewcommand{\@oddhead}{\vbox{\hfill
{\it Lie affgebras on three-dimensional Lie algebras}\hfill
\thepage \hrule}} \makeatother

\label{firstpage}

\section{Introduction}

The study of vector space valued Lie brackets on affine spaces or {\em Lie affgebras} was initiated in \cite{Grabowska1} and developed and applied to the investigation of differential geometry or AV-geometry \cite{Grabowska2, Grabowska3}. In their definition of a Lie affgebra
the authors of \cite{Grabowska1} rely on the existence of a vector space: both the antisymmetry and the Jacobi identity of a Lie bracket are formulated in this vector space. On the
other hand, a vector space-independent definition of an affine space is available.

A new point of view and extension of Lie affgebras was proposed Tomasz Brzezi\'nski and his collaborators in \cite{And, Brzezinski, Brzezinski2}. They gave an intrinsic definition of a Lie affgebra, removes the need for a specified element entirely and formulates the antisymmetry and Jacobi identity of the Lie bracket without invoking a neutral element or a vector space. In this approach, a vector space is an artefact rather than a fundamental ingredient of an affine space in the sense that any point of an affine space determines a vector space, the tangent space or the vector space fibre at this point.

In \cite{Brzezinski}, various examples and properties of Lie affgebras are given, and it is shown how the affine Lie bracket reduces to the linear case. Lie affgebra structures on several classes of affine spaces of matrices are studied in \cite{Brzezinski2}. It is shown that, when retracted to the underlying vector spaces, they correspond to classical matrix Lie algebras: general and special linear, anti-symmetric, anti-hermitian and special antihermitian Lie algebras, respectively. In \cite{And}, it is shown that any Lie affgebra, that is an algebraic system consisting of an affine space together with a bi-affine multiplication satisfying affine versions of the antisymmetry and Jacobi identity, is isomorphic to a Lie algebra together with an element and a specific generalized
derivation. These Lie algebraic data can be taken for the construction of a Lie affgebra or, conversely, they can be uniquely derived for any Lie algebra fibre of the Lie affgebra. It is asserted that a homomorphism between Lie affgebras is given by a homomorphism between Lie algebra fibres and a constant. This allows for the formulation of clear criteria for isomorphisms between Lie affgebras. Using these assertions, the classification of all Lie affgebras with one-dimensional vector space fibres, non-abelian two-dimensional Lie algebra fibres and three-dimensional simple Lie algebra $sl_2$ is given.

In this work, we provide a complete classification of Lie affgebras with three-dimensional non-nilpotent solvable Lie algebras.

\section{Preliminaries}

Let $X$ be a set, $\mathbb{F}$ be a field and $\langle-,-,-\rangle: X^3 \rightarrow X$ ternary operation.

\begin{defn} \cite{Baer}
The set $X$ with the ternary operation  $\langle-,-,-\rangle: X^3 \rightarrow X$ is said to be abelian heap, if for all $x_i\in X$, $i=1,\ldots, 5$,
$$
\langle x_1 ,x_2,x_3\rangle = \langle x_3 ,x_2,x_1\rangle, \; \langle x_1 ,x_1,x_2\rangle = x_2, \; \langle \langle x_1 ,x_2,x_3\rangle, x_4,x_5\rangle = \langle  x_1 ,x_2,\langle  x_3 ,x_4,x_5\rangle\rangle.
$$

\end{defn}

A homomorphism of heaps is a function $f: X\to Y$ preserving the operations in the sense that, $f(\langle x_1 ,x_2,x_3\rangle) = \langle f(x_1) ,f(x_2),f(x_3)\rangle$ for all $x_i\in X.$

\begin{defn}  By an $\mathbb{F}$-affine space, we mean algebraic system $\Big(X, \langle-,-,-\rangle, (-,-,-)\Big),$ where
$\langle-,-,-\rangle: X^3 \rightarrow X$ and $(-,-,-): \mathbb{F} \times X \times X \rightarrow X,$ such that
\begin{itemize}
\item[a)] $(X, \langle-,-,-\rangle)$ is an abelian heap;
\item[b)] For any $\alpha \in \mathbb{F}$ and $a \in X,$ the map $(\alpha, a, -) : X \rightarrow X$ is a homomorphism of heaps;
\item[c)] For a fixed elements $a, b \in X,$ the map $(-, a, b) : \mathbb{F} \rightarrow X$ is a homomorphism of heaps, where $\mathbb{F}$ is the heap with the operation $\alpha - \beta + \gamma;$
\item[d)] For all $\alpha, \beta \in \mathbb{F}$ and $a,b \in X,$ $(\alpha\beta, a, b) = (\alpha, a, (\beta, a, b)),$  $(1, a, b) = b,$ $(0, a, b) = a.$
\item[e)] For all $\alpha \in \mathbb{F}$ and $a,b,c \in X,$ $(\alpha, a, b) = \langle (\alpha, c, b), (\alpha, c, a), a\rangle.$
\end{itemize}
\end{defn}

An affine map $f : X \rightarrow Y$ is a heap homomorphism preserving the actions
in the sense that, for all $a,b,c \in X$ and $\alpha \in \mathbb{F},$
$$f(\alpha, a, b)= (\alpha, f(a), f(b)).$$

The set of affine maps from $X$ to $Y$ is denoted by $\operatorname{Aff}(X, Y).$

Let $X$ be an affine space over $\mathbb{F}$. For a fixed element $e \in X$, we define a binary operation $ +: X \times X \rightarrow X$ and the map
$\mathbb{F} \times X \rightarrow X$ as follows:
$$x+y := \langle x ,e,y\rangle, \quad \alpha a := (\alpha, e, a).$$

Then the triple $(X, +, \alpha)$ forms a vector space,  which is called the \textit{tangent space} to $X$  or the \textit{vector space fibre} of $X$ at the point
$e.$ This tangent space is usually denoted by $T_e(X).$

\begin{defn} \cite{And}
An affine space $X$ with a binary operation $\{-,-\} : X \times X \rightarrow X$ is called a \textit{Lie affgebra}, if the binary operation $\{-,-\}$ satisfies the following
conditions:
 \begin{itemize}

\item[a)] for all $a \in X,$ both $\{a,-\}$ and $\{-,a\}$ are affine map;
\item[b)] affine antisymmetry, that is,
$\langle \{a,b\}, \{a,a\}, \{b,a\}\rangle =  \{b,b\}$ for all $a, b \in X;$
\item[c)] the affine Jacobi identity, that is, for all $a, b, c \in X,$
$$\langle \{a, \{b,c\}\}, \{a,\{a,a\}\}, \{b,\{c,a\}\}, \{b,\{b,b\}\}, \{c,\{a,b\}\}\rangle =  \{c,\{c,c\}\}.$$
\end{itemize}
The multiplication in a Lie affgebra is often referred to as an affine \textit{Lie bracket}.

\end{defn}

It is proven that any tangent space of a Lie affgebra inherits a natural Lie algebra structure.

\begin{thm} \cite{And}
Let $X$ be a Lie affgebra with a bracket $\{-,-\}$. Then, for all $e \in X,$ the tangent space $T_e(X)$
is a Lie algebra with the multiplication
$$[a, b] = \{a, b\} - \{a, e\} + \{e, e\} - \{e, b\}.$$
\end{thm}



 Let $G$ be a Lie algebra. A linear map $f: G \rightarrow G$ is called a generalized derivation in the sense of Leger and Luks
\cite{Leger}, if there exist
linear maps $f', f''$ such that
$$[f(a), b] + [a, f'(b)] = f''([a, b]).$$

In the following theorem, the connection between the Lie affgebras and Lie algebras with the generalized derivation is established.

\begin{thm} \label{thm_Aff1} \cite{And}
 Let $G$ be a Lie algebra and $f, g \in  End(G)$ be such that, for all $a, b \in G$,
\begin{equation} \label{eqfg1} f ([a, b]) = [f(a), b] + [a, f(b)] - [a, g(b)].\end{equation}

Then, $G$ is a Lie affgebra with the affine space structure
\begin{equation} \label{eqfg2}\langle a, b, c \rangle = a-b+c,\quad (\alpha, a, b) = (1-\alpha)a + \alpha b\end{equation}  and the affine Lie bracket (for any fixed $s\in G$)
\begin{equation} \label{eqfg3}\{a, b\} = [a, b] + g(a) + f(b-a) + s.  \end{equation}

We denote this Lie affgebra by $X(G; g, f, s).$

Furthermore, for all $e \in G,$ we have
$T_eX(G; g, f, s) \cong G.$

Conversely, for any Lie affgebra $X$ and any $e \in X,$ there exist $g, f$ necessarily
satisfying \eqref{eqfg1} and $s \in T_eX,$ such that $X = X(T_eX; g, f, s).$

\end{thm}

The criterion for the isomorphism of affgebras $X(G; g, f, s)$ and $X(G; g', f', s')$ is provided in the following theorem.

\begin{thm} \label{thm_Aff2}  \cite{And}
    A Lie affgebras $X(G; g, f, s)$ and $X(G; g', f', s')$ are isomorphic if and only if there exists a Lie algebra automorphism $\Psi: G\to G$ and an element $a\in G$ such that
    \begin{equation}\label{eqfg0} g' = \Psi g \Psi^{-1},\quad
               f' = \Psi (f - \operatorname{ad}_a) \Psi^{-1},\quad
                s' = \Psi(s+ a -g(a)),
            \end{equation}
where $\operatorname{ad}_a$ is an inner derivation, such that $\operatorname{ad}_a(x) = [a,x].$
\end{thm}

\section{Main result}

In this work, we systematically construct all Lie affgebra structures on the three-dimensional complex non-nilpotent solvable Lie algebras.
Here, we give the list of three-dimensional complex non-nilpotent solvable Lie algebras \cite{jac}:
$$\begin{array}{lllll}
\mathbf{r}_3& : & [e_1, e_2]=e_2, & [e_1, e_3]=e_2+e_3,\\[1mm]
\mathbf{r}_3(\lambda)& : & [e_1, e_2]=e_2, & [e_1, e_3]=\lambda e_3,& \lambda \in  \mathbb{C}^*, |\lambda| \leq 1,\\[1mm]
\mathbf{r}_2 \oplus \mathbb{C}& : & [e_1, e_2]=e_2.
\end{array}$$

\subsection{Lie affgebra structures on the algebra $\mathbf{r}_3$}

First, we present the description of the pair of linear transformations
$(f,g),$ that satisfy condition \eqref{eqfg1}.

\begin{prop} Any linear transformations $f$ and $g$ of the algebra $\mathbf{r}_3$ that satisfy condition \eqref{eqfg1} have the following form:
 \begin{equation*}\label{eqfg4}
 \begin{array}{lll}
  f(e_1)=\beta_{1} e_1+\beta_{2} e_2+\beta_{3} e_3, &   f(e_2)=\beta_{4} e_2, & f(e_3)=\beta_{5} e_2+\beta_{4} e_3,\\[1mm]
    g(e_1)=\beta_{1} e_1, & g(e_2)=\beta_{1} e_2, & g(e_3)=\beta_{1} e_3.
\end{array}\end{equation*}
\end{prop}

\begin{proof} The proof follows directly from the definition through a routine verification.
\end{proof}

Applying Theorem \ref{thm_Aff2}, for any element $s\in \mathbf{r}_3,$ we obtain Lie affgebra structure by the binary operation $$\{x, y\} = [x, y] + g(x) + f(y-x) + s.$$


Considering $f-ad_a$ for $a = \beta_5e_1 + (\beta_3-\beta_2)e_2 - \beta_3 e_3,$ instead of $f$, we can easily conclude that Lie affgebra over $\mathbf{r}_3$ is isomorphic to one with

\begin{equation}\label{eqfg5}
 \begin{array}{lll}
  f(e_1)=\beta_{1} e_1, &   f(e_2)=\beta_{4} e_2, & f(e_3)=\beta_{4} e_3,\\[1mm]
    g(e_1)=\beta_{1} e_1, & g(e_2)=\beta_{1} e_2, & g(e_3)=\beta_{1} e_3.
\end{array}
\end{equation}

Thus, for any elements $x=\xi_1e_1 + \xi_2e_2 + \xi_3e_3$ and
$y=\eta_1e_1 + \eta_2 e_2 + \eta_3e_3,$ we obtain an affine Lie bracket
$$\{x, y\} = [x, y] + \beta_1\eta_1 e_1 + \big(\beta_1\xi_2 + \beta_4(\eta_2-\xi_2)\big) e_2 + \big(\beta_1\xi_3 + \beta_4(\eta_3-\xi_3)\big)e_3 + s,$$
where $s = N_1e_1+ N_2e_2+N_3e_3.$
Denote the Lie affgebra with this affine Lie bracket by $F(\beta_{1}, \beta_{4}, N_1, N_2, N_3).$

\begin{prop}\label{prop3.2}
     Two Lie affgebras $F(\beta_{1}, \beta_{4}, N_1, N_2, N_3)$ and
$F(\beta_{1}', \beta_{4}', N_1', N_2', N_3')$ are isomorphic if and only if there exist $\alpha_{1}, \alpha_{2}, \alpha_{3} \in \mathbb{C},$ $\alpha_{4} \in \mathbb{C}^*,$  such that
$$\begin{array}{llll}
    \beta_{1}'=\beta_{1}, & \beta_{4}'=\beta_{4}, & N_2'=\alpha_{1}( N_1 -(\beta_{1}-\beta_{4})(1-\beta_{1}))-\alpha_{2}(\beta_{4}-\beta_{1})(1-\beta_{1})+\alpha_{4}N_2+\alpha_{3} N_3, \\
    N_1'= N_1, &   & N_3'=\alpha_{2}( N_1 -(\beta_{1}-\beta_{4})(1-\beta_{1}))+\alpha_{4} N_3.

\end{array}$$
\end{prop}

\begin{proof} Let $f,g$ and $f',g'$ be linear operators of the form \eqref{eqfg5}. Since $g = \beta_1\operatorname{id},$ it follows from \eqref{eqfg0} that $\beta_1'=\beta_1.$
Since any automorphism of the algebra $\mathbf{r}_3$ has the form
\begin{center}
    $\Psi(e_1)=e_1+\alpha_{1} e_2+\alpha_{2} e_3,$ \quad
$\Psi(e_2)=\alpha_{4} e_2,$ \quad
$\Psi(e_3)=\alpha_{3} e_2 + \alpha_{4} e_3,$
\end{center}
then, using the formulas $$f' = \Psi (f - \operatorname{ad}_a) \Psi^{-1}, \quad s' = \Psi(s+ a -g(a)),$$ for any element $a = C_1e_1+C_2e_2+C_3e_3,$ we obtain
$$\left(\begin{matrix}\beta_1'& 0 & 0 \\[1mm]
0& \beta_4' & 0 \\[1mm]
0& 0 & \beta_4'
\end{matrix}\right) = \left(\begin{matrix}\beta_1& 0 & 0 \\[1mm]
\alpha_{1}(\beta_{1}-\beta_{4}+C_1)+\alpha_{2} C_1 +\alpha_{3} C_3 + \alpha_{4}(C_2+C_3)& \beta_4 - C_1 & -C_1 \\[1mm]
\alpha_{2} (\beta_{1}-\beta_{4}+C_1 )+\alpha_{4}C_3 & 0 & \beta_4 - C_1
\end{matrix}\right)$$
and
 $$\begin{array}{lll} N_1'= C_1 (1-\beta_{1}) +N_1, \\
 N_2'=\alpha_{1} (C_1 (1-\beta_{1}) +N_1) +\alpha_{3} (C_3 (1-\beta_{1}) +N_3)+\alpha_{4} (C_2 (1-\beta_{1}) +N_2), \\
 N_3'=\alpha_{2} (C_1 (1-\beta_{1}) +N_1) +\alpha_{4} (C_3 (1-\beta_{1}) +N_3).
 \end{array}$$

Choosing $$C_1=0, \quad C_3=\frac{\alpha_{2} (\beta_{4}-\beta_{1})}{\alpha_{4}},\quad
C_2=\frac{\alpha_{1}(\beta_{4}- \beta_{1}) -(\alpha_{3}+\alpha_{4})C_3}{\alpha_{4}},$$ yields the required result, thereby completing the proof of the proposition.
\end{proof}

\begin{thm}
Any Lie affgebra structure on the algebra  $\mathbf{r}_3$ is isomorphic to one of the following pairwise non-isomorphic Lie affgebras:
$$\begin{array}{lll}
    F_1(\beta_{1},\beta_{4},N_1,0,0 ),&
    F_2(\beta_{1},\beta_{4},(\beta_{1}-\beta_{4}) (1-\beta_{1} ),0,1),\\[1mm]
    F_3(\beta_{1},\beta_{1},0,1,0),& F_4(1,\beta_{4},0,1,0),\ \beta_4 \neq 1.
\end{array}$$
\end{thm}
\begin{proof}
Using Proposition \ref{prop3.2}, we consider the following cases.
\begin{itemize}
    \item Let $N_1 \neq (\beta_{1}-\beta_{4})(1-\beta_{1}),$ then taking $$\alpha_{1}=\frac{\alpha_{3} N_3+\alpha_{4}N_2 - \alpha_{2} (\beta_{4}-\beta_{1})(1-\beta_{1}) }{(\beta_{1}-\beta_{4})(1-\beta_{1})-N_1}, \quad \alpha_{2}=\frac{\alpha_{4} N_3}{(\beta_{1}-\beta_{4})(1-\beta_{1})- N_1},$$
         we get $N_2'=0$, $N_3'=0$. Hence, the corresponding Lie affgebra is
    $F_1(\beta_{1},\beta_{4},N_1,0,0)$ with $N_1 \neq (\beta_{1}-\beta_{4})(1-\beta_{1}).$

    \item Let $N_1 = (\beta_{1}-\beta_{4})(1-\beta_{1}),$ then $N_2'=\alpha_{2}N_1+\alpha_{4}N_2+\alpha_{3}N_3$ and $N_3'=\alpha_{4}N_3.$
    \begin{itemize}
        \item Let $N_3\neq 0,$ then taking $\alpha_{4} = \frac 1 {N_3},$ $\alpha_{3} = - \frac {\alpha_{2}N_1+\alpha_{4}N_2} {N_3},$
         we obtain  $N_3'=1,$ $N_2'=0.$ Thus, we get the Lie affgebra $F_2(\beta_{1},\beta_{4},(\beta_{1}-\beta_{4}) (1-\beta_{1} ),0,1).$

        \item Let $N_3 = 0,$ then $N_2'=\alpha_{2}N_1+\alpha_{4}N_2.$

        \begin{itemize}
            \item If $N_1\neq 0,$ then $\beta_{1} \neq \beta_{4}$ and $\beta_{1} \neq 1.$ Taking $\alpha_{4} = - \frac{\alpha_{2}N_1}{N_2},$ we get $N_2'=0$ and obtain the Lie affgebra
            $F_1(\beta_{1},\beta_{4},N_1,0,0)$ with $N_1 =(\beta_{1}-\beta_{4})(1-\beta_{1}) \neq 0.$
            \item If $N_1= 0,$ then $N_2'=\alpha_{4}N_2$, and $\beta_{1} = \beta_{4}$ or $\beta_{1} = 1.$
           \begin{itemize}
            \item If $N_2 = 0,$ then we get the Lie affgebras
            $F_1(\beta_{1},\beta_{1},0,0,0)$ and  $F_1(1,\beta_{4},0,0,0).$
            \item If $N_2 \neq 0,$ then taking $\alpha_{4} = \frac 1 {N_2},$ we have $N_2'=1$ and
             obtain the Lie affgebras
             $F_3(\beta_{1},\beta_{1},0,1,0)$ and $F_4(1,\beta_{4},0,1,0).$
        \end{itemize}
        \end{itemize}

    \end{itemize}
\end{itemize}
\end{proof}

\subsection{Lie affgebra structures on the algebra $\mathbf{r}_3(\lambda)$}

In the following proposition, we present the description of the pair of linear transformations
$(f,g),$ for the algebra $\mathbf{r}_3(\lambda),$ that satisfy condition \eqref{eqfg1}.

\begin{prop} Any linear transformations $f$ and $g$ of the algebra $\mathbf{r}_3(\lambda)$ that satisfy condition \eqref{eqfg1} have the following form:
 $$\begin{array}{lllll}
   \lambda\neq 1: & f(e_1)=\beta_{1} e_1+\beta_{2} e_2+\beta_{3} e_3, &   f(e_2)=\beta_{4} e_2, & f(e_3)=\beta_{5} e_3,\\[1mm]
    & g(e_1)=\beta_{1} e_1, & g(e_2)=\beta_{1} e_2, & g(e_3)=\beta_{1} e_3.\\[2mm]
  \lambda= 1: & f(e_1)=\beta_{1} e_1+\beta_{2} e_2+\beta_{3} e_3, &   f(e_2)=\beta_{4} e_2 + \beta_{6} e_3, & f(e_3)=\beta_{7} e_2+\beta_{5} e_3,\\[1mm]
    & g(e_1)=\beta_{1} e_1, & g(e_2)=\beta_{1} e_2, & g(e_3)=\beta_{1} e_3.
\end{array}$$
\end{prop}

\begin{proof} The proof of the proposition follows directly from the straightforward verification.
\end{proof}


It is not difficult to get that, any automorphism of the algebra $\mathbf{r}_3(\lambda)$ has the form
$$\begin{array}{lllll}
    \lambda\neq 1: & \Psi(e_1)=e_1+\alpha_{1} e_2+\alpha_{2} e_3,& \Psi(e_2)=\alpha_{3} e_2,& \Psi(e_3)=\alpha_{4} e_3.\\[1mm]
    \lambda= 1: & \Psi(e_1)=e_1+\alpha_{1} e_2+\alpha_{2} e_3,& \Psi(e_2)=\alpha_{3} e_2+\alpha_{5} e_3,& \Psi(e_3)=\alpha_{4} e_3 + \alpha_{6} e_2.\\[1mm]
\end{array}$$

\textbf{Case {$\lambda \neq1$}.}
 Considering $f-ad_a$ for $a = \beta_4e_1 -\beta_2e_2 - \frac{\beta_3}{\lambda} e_3,$ instead of $f$, we can easily conclude that Lie affgebra over $\mathbf{r}_3(1)$ is isomorphic to one with
\begin{equation}\label{eqfg8}
 \begin{array}{lll}
  f(e_1)=\beta_{1} e_1, &   f(e_2)=0, & f(e_3)=\beta_{5} e_3,\\[1mm]
    g(e_1)=\beta_{1} e_1, & g(e_2)=\beta_{1} e_2, & g(e_3)=\beta_{1} e_3.
\end{array}
\end{equation}

Thus, for any elements $x=\xi_1e_1 + \xi_2e_2 + \xi_3e_3$ and
$y=\eta_1e_1 + \eta_2 e_2 + \eta_3e_3,$ we obtain an affine Lie bracket
$$\{x, y\} = [x, y] + \beta_1 \eta_1e_1 + \beta_1 \xi_2e_2 +(\beta_1 \xi_3 + \beta_5 \eta_3 - \beta_5 \xi_3)e_3 + s,$$
where $s = N_1e_1+ N_2e_2+N_3e_3.$
Denote the Lie affgebra on the algebra $\mathbf{r}_3(\lambda)$ with this affine Lie bracket by $H(\beta_{1}, \beta_{5}, N_1, N_2, N_3).$

\begin{prop}\label{prop3.22}
     Two Lie affgebras $H(\beta_{1}, \beta_{5}, N_1, N_2, N_3)$ and
$H(\beta_{1}', \beta_{5}', N_1', N_2', N_3')$ are isomorphic if and only if there exist $\alpha_{1}, \alpha_{2}\in \mathbb{C},$ $\alpha_{3}, \alpha_{4} \in \mathbb{C}^*,$ such that
$$\begin{array}{lllll}
    \beta_{1}'=\beta_{1} , & \beta_{5}'=\beta_{5} , & N_2'=\alpha_{3} N_2 + \alpha_{1}\big(N_1 -  \beta_{1}(1-\beta_{1})\big), \\
    N_1'= N_1, & & N_3'=\alpha_{4} N_3 + \alpha_{2} \big( N_1 - \frac{(\beta_{1}-\beta_{5})(1-\beta_{1})} {\lambda} \big).
\end{array}$$

\end{prop}

\begin{proof} The proof is straightforward and follows similarly to the proof of Proposition \ref{prop3.2}.
\end{proof}

%
%
%
%
%
%
%

\begin{prop} Any Lie affgebra structure on the algebra  $\mathbf{r}_3(\lambda)$ with $\lambda\neq 1,$ is isomorphic to one of the following pairwise non-isomorphic Lie affgebras:
$$\begin{array}{lll}
    H_1(\beta_{1},\beta_{5},N_1,0,0), &H_2(\beta_{1},\beta_{5},\beta_{1}(1-\beta_{1}),1,0),\\[1mm]
    H_3(\beta_{1},\beta_{5},\frac{(\beta_{1}-\beta_{5})(1-\beta_{1})} {\lambda},0,1), &
H_4(1,\beta_{5},0,1,1), \ \beta_5 \neq 1-\lambda,\\[1mm]
    H_5(\beta_{1},\beta_{1}(1-\lambda),\beta_{1}(1-\beta_{1}),1,1 ).
\end{array}$$
\end{prop}

\begin{proof} Using Proposition \ref{prop3.22}, we consider the following cases.
\begin{itemize}
    \item  $N_1\neq \beta_{1}(1-\beta_{1})$ and $N_1 \neq \frac{(\beta_{1}-\beta_{5})(1-\beta_{1})} {\lambda},$ then taking
    $\alpha_{1}=\frac{\alpha_{3} N_2}{\beta_{1}(1-\beta_{1}) - N_1},$ $\alpha_{2}=\frac{\alpha_{4} \lambda N_3}{(\beta_{1}-\beta_{5})(1-\beta_{1})-\lambda N_1},$ we get $N_2'=N_3'=0$ and obtain  the Lie affgebra    $H_1(\beta_{1},\beta_{5},N_1,0,0).$
    \item  $N_1= \beta_{1}(1-\beta_{1})$ and $N_1 \neq \frac{(\beta_{1}-\beta_{5})(1-\beta_{1})} {\lambda},$ then we have $\beta_1 \neq 1,$ $\beta_5 \neq \beta_1(1-\lambda)$ and taking
    $\alpha_{2}=\frac{\alpha_{4} \lambda N_3}{(\beta_{1}-\beta_{5})(1-\beta_{1})-\lambda N_1},$ we obtain $N_3'=0$ and $N_2'=\alpha_{3} N_2.$
    Thus, in the case of $N_2 =0,$ we obtain the Lie affgebra $H_1$ with $N_1= \beta_{1}(1-\beta_{1}).$ In the case of $N_2 \neq 0,$ we can get $N_2'=1$  and obtain the Lie affgebra $H_2(\beta_{1},\beta_{5},\beta_{1}(1-\beta_{1}),1,0).$

    \item  $N_1\neq \beta_{1}(1-\beta_{1})$ and $N_1 = \frac{(\beta_{1}-\beta_{5})(1-\beta_{1})} {\lambda},$ then taking
    $\alpha_{1}=\frac{\alpha_{3} N_2}{\beta_{1}(1-\beta_{1}) - N_1},$ we obtain $N_2'=0$ and $N_3'=\alpha_{4} N_3.$
    Thus, in the case of $N_3 =0,$ we obtain the affgebra $H_2$ with $N_1= \frac{(\beta_{1}-\beta_{5})(1-\beta_{1})} {\lambda}.$ In the case of $N_3 \neq 0,$ we can suppose $N_3'=1$  and obtain the affgebra $H_3(\beta_{1},\beta_{5},\frac{(\beta_{1}-\beta_{5})(1-\beta_{1})} {\lambda},0,1).$

    \item  $N_1= \beta_{1}(1-\beta_{1})$ and $N_1 = \frac{(\beta_{1}-\beta_{5})(1-\beta_{1})} {\lambda}.$ Then we have $\beta_1=1$ or $\beta_5 = \beta_1 (1-\lambda)$ and
$N_2'=\alpha_{3} N_2,$ $N_3'=\alpha_{4} N_3.$ If $N_2N_3=0,$ then we obtain one of the affgebras from $H_1, H_2, H_3.$ In the case of $N_2N_3\neq0,$ we can suppose $N_2'=N_3'=1$ and obtain the affgebras  $H_4(1,\beta_{5},0,1,1)$ and $H_5(\beta_{1},\beta_{1}(1-\lambda),\beta_{1}(1-\beta_{1}),1,1 ).$
\end{itemize}

\end{proof}

\textbf{Case {$\lambda =1$}.}
 Considering $f-ad_a$ for $a = \mu e_1 -\beta_2e_2 - \beta_3e_3,$ instead of $f$, we can easily conclude that Lie affgebra over $\mathbf{r}_3(\lambda)$ is isomorphic to one with
\begin{equation*}\label{eqfg8.1}
 \begin{array}{lll}
  f(e_1)=\beta_{1} e_1, &   f(e_2)=(\beta_{4} - \mu) e_2 + \beta_{6} e_3, & f(e_3)=\beta_{7} e_2+(\beta_{5}-\mu) e_3,\\[1mm]
    g(e_1)=\beta_{1} e_1, & g(e_2)=\beta_{1} e_2, & g(e_3)=\beta_{1} e_3.
\end{array}
\end{equation*}

By selecting $\mu$ as one of the eigenvalues of the matrix $\left(\begin{matrix} \beta_{4} & \beta_{6} \\ \beta_{7} & \beta_{5}\end{matrix}\right),$ without loss of generality we can assume that the matrix has a zero eigenvalue.

From the formulas $f' = \Psi (f - \operatorname{ad}_a) \Psi^{-1},$ and the general form of the automorphism $\Psi,$ it follows that the matrix of the operator $f'$ can be reduced to one of the following Jordan forms:
$$\left(\begin{matrix}\beta_1 & 0& 0 \\ 0&  0 & 0 \\ 0& 0 & \beta_{5}\end{matrix}\right), \quad \left(\begin{matrix}\beta_1 & 0& 0 \\ 0&  0 & \beta_6 \\ 0& 0 & 0\end{matrix}\right), \ \beta_6\neq0.$$

If the Jordan form of the matrix corresponds to the first case, then the situation is analogous to that of $\lambda \neq 1$ and we obtain the affgebras $H_1, H_2, H_3, H_4$ and $H_5$ with $\lambda=1.$

If the Jordan form of the matrix corresponds to the second case, then we get
\begin{equation}
 \begin{array}{lll}
  f(e_1)=\beta_{1} e_1, &   f(e_2)= \beta_{6} e_3, & f(e_3)=0,\\[1mm]
    g(e_1)=\beta_{1} e_1, & g(e_2)=\beta_{1} e_2, & g(e_3)=\beta_{1} e_3,
\end{array}
\end{equation}
and for any elements $x=\xi_1e_1 + \xi_2e_2 + \xi_3e_3,$
$y=\eta_1e_1 + \eta_2 e_2 + \eta_3e_3,$ we obtain an affine Lie bracket
$$\{x, y\} = [x, y] + \beta_1 \eta_1e_1 + \beta_1 \xi_2e_2 +(\beta_1 \xi_3 + \beta_6 \eta_2 - \beta_6 \xi_2)e_3 + s,$$
where $s = N_1e_1+ N_2e_2+N_3e_3.$
Denote the Lie affgebra on the algebra $\mathbf{r}_3(\lambda)$ with this affine Lie bracket by $K(\beta_{1}, \beta_{6}, N_1, N_2, N_3).$

\begin{prop}\label{prop3.222}
     Two Lie affgebras $K(\beta_{1}, \beta_{6}, N_1, N_2, N_3)$ and
$K(\beta_{1}', \beta_{6}', N_1', N_2', N_3')$ are isomorphic if and only if there exist $\alpha_{1}, \alpha_{2}\in \mathbb{C},$ $\alpha_{3}, \alpha_{4} \in \mathbb{C}^*,$  such that
$$\begin{array}{lllll}
    \beta_{1}'=\beta_{1} , & \beta_{6}'=\frac{\alpha_4 \beta_6}{\alpha_3}, & N_2'=\alpha_{3} N_2 + \alpha_{1}\big(N_1 -  \beta_{1}(1-\beta_{1})\big), \\
    N_1'= N_1, & & N_3'=\alpha_{4} N_3 + \alpha_{2}\big(N_1 -  \beta_{1}(1-\beta_{1})\big) + \alpha_5N_2
+\frac{\alpha_1 \alpha_4 \beta_6 (1-\beta_1)}{\alpha_3}.
\end{array}$$
\end{prop}

\begin{proof} The proof is straightforward and follows similarly to the proof of Proposition \ref{prop3.2}.
\end{proof}

\begin{prop} Any Lie affgebra from the class $K(\beta_{1}, \beta_{6}, N_1, N_2, N_3)$ is isomorphic to one of the following pairwise non-isomorphic Lie affgebras:
$$\begin{array}{lll}
    K_1(\beta_{1},1,N_1,0,0), &K_2(\beta_{1},1,\beta_{1}(1-\beta_{1}),1,0),& K_3(1,1,0,0,1).
\end{array}$$
\end{prop}

\begin{proof} By taking $\alpha_3= \alpha_4\beta_6,$ in Proposition \ref{prop3.222}, without loss of generality we may assume
$\beta_6=1$ and $\alpha_3=\alpha_4.$ We now consider the following cases:
\begin{itemize}
    \item  $N_1\neq \beta_{1}(1-\beta_{1}),$ then taking
    $\alpha_{1}=\frac{\alpha_{3} N_2}{\beta_{1}(1-\beta_{1}) - N_1},$ $\alpha_{2}=\frac{\alpha_{4} N_3 + \alpha_{5} N_2+ \alpha_1(1-\beta_1)}{\beta_{1}(1-\beta_{1}) - N_1},$ we get $N_2'=N_3'=0$ and obtain the affgebra
    $K_1(\beta_{1},1,N_1,0,0).$
    \item  $N_1= \beta_{1}(1-\beta_{1}),$ then $N_2'=\alpha_{4} N_2,$ $N_3'=\alpha_{4} N_3 + \alpha_5N_2
+\alpha_1 (1-\beta_1).$
    \begin{itemize}
    \item  If $N_2\neq 0,$ then choosing $\alpha_{4}=\frac{1} {N_2},$ $\alpha_{5}=-\frac{\alpha_{4} N_3
+\alpha_1 (1-\beta_1)} {N_2},$ we have $N_2'=1,$ $N_3'=0$ and obtain the affgebra $K_2(\beta_{1},1,\beta_{1}(1-\beta_{1}),1,0).$
    \item  If $N_2= 0,$ then $N_3'=\alpha_{4} N_3 +\alpha_1 (1-\beta_1).$ If $\beta_1\neq 1$ or $N_3=0,$ then by choosing appropriate values for $\alpha_1$ and $\alpha_4,$ we have $N_3'=0$ and obtain the affgebra $K_1$ with $N_1= \beta_{1}(1-\beta_{1}).$ If
 $\beta_1=1$ and $N_3\neq 0,$ then taking $\alpha_{4}=\frac{1} {N_3},$  we have $N_3'=1$ and obtain the affgebra $K_3(1,1,0,0,1).$
    \end{itemize}
\end{itemize}

\end{proof}

Summarizing the results for the cases $\lambda\neq 1$ and $\lambda=1,$ we obtain the following theorem.

\begin{thm} Any Lie affgebra structure on the algebra  $\mathbf{r}_3(\lambda)$ is isomorphic to one of the following pairwise non-isomorphic Lie affgebras:
$$\begin{array}{lll}
    H_1(\beta_{1},\beta_{5},N_1,0,0), &H_2(\beta_{1},\beta_{5},\beta_{1}(1-\beta_{1}),1,0),\\[1mm]
    H_3(\beta_{1},\beta_{5},\frac{(\beta_{1}-\beta_{5})(1-\beta_{1})} {\lambda},0,1), &
H_4(1,\beta_{5},0,1,1), \ \beta_5 \neq 1-\lambda,\\[1mm]
    H_5(\beta_{1},\beta_{1}(1-\lambda),\beta_{1}(1-\beta_{1}),1,1 ),
\end{array}$$
and
$$\begin{array}{lll}
    K_1(\beta_{1},1,N_1,0,0), &K_2(\beta_{1},1,\beta_{1}(1-\beta_{1}),1,0),& K_3(1,1,0,0,1).
\end{array}$$
\end{thm}

Note that in the case of $\lambda=1$, we have $H_2(\beta_{1},0,\beta_{1}(1-\beta_{1}),1,0) \simeq H_3(\beta_{1},0,\beta_{1}(1-\beta_{1}),0,1).$

\subsection{Lie affgebra structures on the algebra $\mathbf{r}_2 \oplus \mathbb{C}$}

First, we present the description of the pair of linear transformations
$(f,g),$ that satisfy condition \eqref{eqfg1}.

\begin{prop}
    Any linear transformations $f$ and $g$ of the algebra $\mathbf{r}_2 \oplus \mathbb{C},$ that satisfy condition \eqref{eqfg1} have the following form:
 $$\begin{array}{lll}
  f(e_1)=\beta_{1} e_1+\beta_{2} e_2+\beta_{3} e_3, &   f(e_2)=\beta_{4} e_2, & f(e_3)=\beta_{5} e_3,\\[1mm]
    g(e_1)=\beta_{1} e_1+\gamma_{1} e_3, & g(e_2)=\beta_{1} e_2+\gamma_{2} e_3, & g(e_3)=\gamma_{3} e_3.
\end{array}$$
\end{prop}

\begin{proof} The proof of the proposition follows directly from the straightforward verification.
\end{proof}

For any element $s= N_1e_1+N_2e_2+N_3e_3\in \mathbf{r}_2 \oplus \mathbb{C},$ a Lie affgebra structure is defined by the binary operation $$ \{x, y\} = [x, y] + g(x) + f(y-x) + s,$$
which is denoted by $X(\mathbf{r}_2 \oplus \mathbb{C}; g, f, s).$

Since any automorphism of the algebra $\mathbf{r}_2 \oplus \mathbb{C}$ has the form
\begin{center}
    $\Psi(e_1)=e_1+\alpha_{1} e_2+\alpha_{2} e_3,$ \quad
$\Psi(e_2)=\alpha_{3} e_2,$ \quad
$\Psi(e_1)=\alpha_{4} e_3,$
\end{center}
we obtain the following proposition.

\begin{prop}\label{pr3.5} Two Lie affgebras $X(\mathbf{r}_2 \oplus \mathbb{C}; g, f, s)$ and
$X\mathbf{r}_2 \oplus \mathbb{C}; g', f', s')$ are isomorphic if and only if there exist $\alpha_{1}, \alpha_{2}, C_1, C_2, C_3\in \mathbb{C}$ and  $\alpha_{3}, \alpha_{4} \in \mathbb{C}^*,$ such that
$$\begin{array}{lll}
    \beta_{1}'=\beta_{1}, &  \beta_{2}'=\alpha_{1}(\beta_{1}+\beta_{4}-C_1)+\alpha_{3} (\beta_{2} +C_2), & \beta_{3}'=\alpha_{2} (\beta_{1}-\beta_{5})+\alpha_{4}
 \beta_{3}, \\
 \beta_{5}'=\beta_{5}, & \beta_{4}'= \beta_{4} -C_1 , & \\
 \gamma_{3}'=\gamma_{3}, & \gamma_{1}'=\alpha_{2} (\beta_{1}-\gamma_{3})+\alpha_{4} (\gamma_{1}-\frac{\alpha_{1} \gamma_{2}}{\alpha_{3}}), & \gamma_{2}'=\frac{\alpha_{4} \gamma_{2}}{\alpha_{3}},
\end{array}$$
and
 $$\begin{array}{lll} N_1'= C_1 (1-\beta_{1}) +N_1, \\
 N_2'=\alpha_{1} (C_1 (1-\beta_{1}) +N_1) +\alpha_{3} (C_2 (1-\beta_{1}) +N_2), \\
 N_3'=\alpha_{2} (C_1 (1-\beta_{1}) +N_1) +\alpha_{4} (C_3 (1-\gamma_{3})+C_1 \gamma_{1} + C_2 \gamma_{2} +N_3).
 \end{array}$$
\end{prop}

\begin{proof} The proof is obtained by straightforward computation using Theorem\ref{thm_Aff2}.

\end{proof}

Choosing $C_1=\beta_{4}$ and
$C_2=-\frac{\alpha_{3} \beta_{2}+\alpha_{1} \beta_{1}}{\alpha_{3}}$ in Proposition \ref{pr3.5}, we obtain $\beta_{2}'= \beta_{4}'=0.$ Therefore, without loss of generality, we may assume $\beta_{2}= \beta_{4}=0,$ $C_1=0,$ $C_2=-\frac{\alpha_{1} \beta_{1}}{\alpha_{3}}.$ Hence, we derive the following restrictions:
\begin{equation}\label{eq3.1}\begin{array}{lll}
    \beta_{3}'=\alpha_{2} (\beta_{1}-\beta_{5})+\alpha_{4}
 \beta_{3}, &
 \gamma_{1}'=\alpha_{2} (\beta_{1}-\gamma_{3})+\alpha_{4} (\gamma_{1}-\frac{\alpha_{1} \gamma_{2}}{\alpha_{3}}), & \gamma_{2}'=\frac{\alpha_{4} \gamma_{2}}{\alpha_{3}},
\end{array}\end{equation}
and
 \begin{equation}\label{eq3.2}\begin{array}{lll} N_1'= N_1, \\
 N_2'=\alpha_{3} N_2 +\alpha_{1} \big( N_1 - \beta_1(1-\beta_{1})\big), \\
 N_3'=\alpha_{4}N_3 + \alpha_{2} N_1 +\alpha_{4} \big(C_3 (1-\gamma_{3})-\frac{\alpha_{1} \beta_{1}}{\alpha_{3}} \gamma_{2}\big).
 \end{array}\end{equation}

Hence, Lie affgebra structures on the algebra $\mathbf{r}_2 \oplus \mathbb{C}$  depend on the parameters $\beta_{1},\beta_{3}, \beta_{5}, \gamma_{1}, \gamma_{2}, \gamma_{3}, N_1, N_2, N_3.$
We denote this class of Lie affgebras by $$L(\beta_{1},\beta_{3}, \beta_{5}, \gamma_{1}, \gamma_{2}, \gamma_{3}, N_1, N_2, N_3).$$

\begin{thm} Any Lie affgebra structure on the algebra $\mathbf{r}_2 \oplus \mathbb{C}$ is isomorphic to one of the following
pairwise non-isomorphic Lie affgebras:
$$\begin{array}{lll}
L_1(\beta_{1},0,\beta_{5},0,1,\gamma_{3},N_1,0,0), &
L_2(\beta_{1},0,\beta_{5},0,1,\gamma_{3},N_1,1,0),\\[1mm]
L_3(\beta_{1},0,\beta_{5},0,1,1,N_1,N_2,1),&
L_4(\beta_{1},0,\beta_{5},0,0,\gamma_{3},N_1,0,0),\\[1mm]
L_5(\beta_{1},0,\beta_{5},1,0,\gamma_{3},N_1,0,0),&
L_6(\beta_{1},0,\beta_{5},\gamma_{1},0,1,N_1,0,1),\\[1mm]
L_7(\beta_{1},0,\beta_{5},0,0,\gamma_{3}, \beta_{1} (1-\beta_{1}),1,0),&
L_8(\beta_{1},0,\beta_{5},1,0,\gamma_{3}, \beta_{1} (1-\beta_{1}),1,0), \\[1mm]
L_9(\beta_{1},0,\beta_{5},\gamma_{1},0,1, \beta_{1} (1-\beta_{1}),1,1),&
L_{10}(\beta_{1},1,\beta_{1},0,1,\gamma_{3},N_1,0,0), \\[1mm]
L_{11}(\beta_{1},1,\beta_{1},0,1,\beta_{1},N_1,N_2,0), \beta_1\neq 1,&
L_{12}(\beta_{1},1,\beta_{1},0,1,\gamma_{3},\beta_{1}(1-\beta_{1}),N_2,0), \ \gamma_3 \neq 1,\\[1mm]
L_{13}(\beta_{1},1,\beta_{1},0,1,1,N_1,N_2,0),& L_{14}(\beta_{1},\beta_{3},\beta_{1},0,1,1,\beta_{1}(1-\beta_{1}),N_2,1),\\[1mm]
L_{15}(\beta_{1},1,\beta_{1},0,0,\gamma_{3},N_1,0,0),&
L_{16}(\beta_{1},\beta_{3},\beta_{1},1,0,\beta_{1},N_1,0,0), \ \beta_1 \neq 1,\\[1mm]
L_{17}(\beta_{1},\beta_{3},\beta_{1},1,0,1,N_1,0,0), &
L_{18}(\beta_{1},\beta_{3},\beta_{1},0,0,1,0,0,1), \\[1mm]
L_{19}(\beta_{1},1,\beta_{1},0,0,\gamma_{3},\beta_{1} (1-\beta_{1}),1,0),&
L_{20}(\beta_{1},\beta_{3},\beta_{1},1,0,\beta_{1},\beta_{1} (1-\beta_{1}),1,0), \ \beta_1\neq 1,\\[1mm]
L_{21}(\beta_{1},\beta_{3},\beta_{1},0,0,1,\beta_{1} (1-\beta_{1}),0,1),&
L_{22}(\beta_{1},\beta_{3},\beta_{1},0,0,1,\beta_{1} (1-\beta_{1}),1,1),\\[1mm]
L_{23}(1,\beta_{3},1,1,0,1,0,0,N_3),&
L_{24}(1,\beta_{3},1,1,0,1,0,1,N_3)
\end{array}$$
\end{thm}

\begin{proof}
Consider the following cases.
\begin{itemize}
    \item Let $\beta_{1} \neq \beta_{5}$ and $\gamma_{2} \neq 0,$ then taking $\alpha_{2} =\frac{\alpha_{4} \beta_{3}}{\beta_{5}-\beta_{1}},$ $\alpha_{4}=\frac{\alpha_{3}}{\gamma_{2}}$ and $\alpha_{1}=\alpha_{2}(\beta_{1}-\gamma_{3})+ \alpha_{4} \gamma_{1},$ we obtain $\beta_{3}'=0,$ $\gamma_{1}'=0,$  $\gamma_{2}'=1.$ Thus, we get that
    $N_2'=\alpha_{3} N_2,$ $ N_3'=\alpha_{3} (C_3 (1-\gamma_{3}) + N_3).$
        \begin{itemize}
            \item Let $\gamma_{3} \neq 1,$ then taking
            $C_3=\frac{N_3}{\gamma_{3}-1},$ we obtain  $N_3'=0.$
            Hence, we get the Lie affgebras $L_1(\beta_{1},0,\beta_{5},0,1,\gamma_{3},N_1,0,0)$ and $L_2(\beta_{1},0,\beta_{5},0,1,\gamma_{3},N_1,1,0)$ depending on whether $N_2 = 0$ or not.
            \item Let $\gamma_{3} = 1,$ then we get $N_2'=\alpha_{3} N_2,$ $N_3'=\alpha_{3} N_3.$
            \begin{itemize}
            \item If $N_3 \neq 0,$ then taking $\alpha_{3}=\frac{1}{N_3},$ we get $N_3'=1,$ and obtain the affgebra $L_3(\beta_{1},0,\beta_{5},0,1,1,N_1,N_2,1).$
            \item If $N_3 = 0,$ then we obtain the
            affgebras $L_1(\beta_{1},0,\beta_{5},0,1,1,N_1,0,0)$ and $L_2(\beta_{1},0,\beta_{5},0,1,1,N_1,1,0)$ depending on whether $N_2 = 0$ or not.
            \end{itemize}

        \end{itemize}

\item Let $\beta_{1} \neq \beta_{5}$ and $\gamma_{2} = 0,$ then $\gamma_{2}'=0$ and taking $\alpha_{2} =\frac{\alpha_{4} \beta_{3}}{\beta_{5}-\beta_{1}},$   we obtain $\beta_{3}'=0.$ Thus, we get that
    $$\gamma_{1}'=\alpha_{4} \gamma_{1}, \quad N_2'=\alpha_{1} (N_1 - \beta_{1} (1-\beta_{1}))+\alpha_{3} N_2, \quad N_3'=\alpha_{4} (C_3 (1-\gamma_{3})+N_3).$$

\begin{itemize}
        \item Let $N_1 \neq \beta_{1} (1-\beta_{1})$ and $\gamma_{3} \neq 1,$ then taking
            $\alpha_{1}=\frac{\alpha_{3} N_2}{\beta_{1} (1-\beta_{1}) -N_1},$ $C_3=\frac{N_3}{\gamma_{3}-1},$ we obtain     $N_2'=0,$ $N_3'=0$ and $\gamma_{1}'=\alpha_{4} \gamma_{1}.$ Thus, in this case we get the affgebras
             $L_4(\beta_{1},0,\beta_{5},0,0,\gamma_{3},N_1,0,0)$ and
              $L_5(\beta_{1},0,\beta_{5},1,0,\gamma_{3},N_1,0,0)$ depending on whether $\gamma_{1} = 0$ or not.
 \item Let $N_1 \neq \beta_{1} (1-\beta_{1})$ and $\gamma_{3} =1,$ then taking
            $\alpha_{1}=\frac{\alpha_{3} N_2}{\beta_{1} (1-\beta_{1})-N_1},$ we obtain  $N_2'=0,$ and $N_3'=\alpha_{4} N_3,$ $\gamma_{1}'=\alpha_{4} \gamma_{1}.$
            \begin{itemize}
               \item If $N_3=0,$ then $N'_3=0,$ and we get the affgebras
             $L_4(\beta_{1},0,\beta_{5},0,0,1,N_1,0,0)$ and
              $L_5(\beta_{1},0,\beta_{5},1,0,1,N_1,0,0)$ depending on whether $\gamma_{1} = 0$ or not.
              \item If $N_3\neq 0,$ then taking $\alpha_{4} = \frac 1{N_3},$ we get $N_3'=1$ and obtain the affgebra
              $L_6(\beta_{1},0,\beta_{5},\gamma_{1},0,1,N_1,0,1).$

            \end{itemize}
\item Let $N_1 = \beta_{1} (1-\beta_{1})$ and $\gamma_{3} \neq 1,$ then taking
            $C_3=\frac{N_3}{\gamma_{3}-1},$ we obtain $N_3'=0$ and $N_2'=\alpha_{3}N_2,$ $\gamma_{1}'=\alpha_{4} \gamma_{1}.$
           \begin{itemize}
               \item If $N_2=0,$ then $N'_2=0,$ and we get the affgebras
             $L_4(\beta_{1},0,\beta_{5},0,0,1,\beta_{1} (1-\beta_{1}),0,0)$ and
              $L_5(\beta_{1},0,\beta_{5},1,0,1,\beta_{1} (1-\beta_{1}),0,0)$ depending on whether $\gamma_{1} = 0$ or not.
              \item If $N_2\neq 0,$ then taking $\alpha_{3} = \frac 1{N_2},$ we get $N_2'=1$ and obtain the affgebras
              $L_7(\beta_{1},0,\beta_{5},0,0,\gamma_{3}, \beta_{1} (1-\beta_{1}),1,0)$ and $L_8(\beta_{1},0,\beta_{5},1,0,\gamma_{3}, \beta_{1} (1-\beta_{1}),1,0)$ depending on whether $\gamma_{1} = 0$ or not.
            \end{itemize}
\item Let $N_1 = \beta_{1} (1-\beta_{1})$ and $\gamma_{3} = 1,$ then we get $N_2'=\alpha_{3}N_2,$ $N_3'=\alpha_{4}N_3,$   $\gamma_{1}'=\alpha_{4} \gamma_{1}.$
        \begin{itemize}
               \item If $N_2=0,$ $N_3=0,$ then $N_2'=0,$ $N_3'=0,$ and we obtain the affgebras
               $L_4(\beta_{1},0,\beta_{5},0,0,1,\beta_{1} (1-\beta_{1}),0,0)$ and
              $L_5(\beta_{1},0,\beta_{5},1,0,1,\beta_{1} (1-\beta_{1}),0,0)$ depending on whether $\gamma_{1} = 0$ or not.
        \item If $N_2=0,$ $N_3\neq0,$ then $N_2'=0,$ and taking $\alpha_{4} = \frac 1{N_3},$ we get $N_3'=1.$ Thus, in this case we obtain the affgebra $L_6(\beta_{1},0,\beta_{5},\gamma_{1},0,1,\beta_{1} (1-\beta_{1}),0,1).$
        \item If $N_2\neq 0,$ $N_3=0,$ then $N_3'=0,$ and taking $\alpha_{3} = \frac 1{N_2},$ we get $N_2'=1.$ Thus, in this case, we obtain the affgebras $L_7(\beta_{1},0,\beta_{5},0,0,1, \beta_{1} (1-\beta_{1}),1,0),$ and
$L_8(\beta_{1},0,\beta_{5},1,0,1, \beta_{1} (1-\beta_{1}),1,0)$
depending on whether $\gamma_{1} = 0$ or not.
       \item If $N_2\neq 0,$ $N_3\neq0,$  taking $\alpha_{3} = \frac 1{N_2},$  $\alpha_{4} = \frac 1{N_3},$ we get $N_2'=N_3'=1.$  Hence, we get the affgebra $L_9(\beta_{1},0,\beta_{5},\gamma_{1},0,1, \beta_{1} (1-\beta_{1}),1,1).$
            \end{itemize}

\end{itemize}

    \item Let  $\beta_{1} = \beta_{5}$  and $\gamma_{2} \neq 0,$ then taking $\alpha_{1} =\alpha_{4} \gamma_{1} + \alpha_{2}( \beta_{1}-\gamma_{3})$ and
    $\alpha_{4} = \frac{\alpha_{3}}{\gamma_{2}},$
 we obtain   $\gamma_{1}'=0,$  $\gamma_{2}'=1.$
        Thus, without loss of generality, we may assume $\gamma_1=0,$ $\gamma_2=1,$ then $\alpha_{1} =\alpha_{2}( \beta_{1}-\gamma_{3}),$ $\alpha_{4} = \alpha_{3},$ and we have
    $$\beta_{3}'=\alpha_{3}\beta_{3},\quad N_2'=\alpha_{3} N_2+\alpha_{2} (\beta_{1}-\gamma_{3}) (N_1-\beta_{1}(1-\beta_{1})),$$ $$N_3'=\alpha_{3} (C_3 (1-\gamma_{3}) + N_3) + \alpha_{2} (N_1-\beta_{1}(\gamma_{3}-\beta_{1})) .$$

    \begin{itemize}
        \item Let $\gamma_{3} \neq 1$ and $(\beta_{1}-\gamma_{3})(N_1-\beta_{1}(1-\beta_{1}))\neq 0,$ then taking  $$C_3=\frac{ \alpha_{3}N_3+\alpha_{2} (N_1-\beta_{1}(\gamma_{3}-\beta_{1}))}{\alpha_{3}  (\gamma_{3}-1)},\quad \alpha_{2}=\frac{\alpha_{3} N_2}{(\gamma_{3}-\beta_{1}) (N_1-\beta_{1}(1-\beta_{1}))},$$ we obtain $N_2'=0,$ $N_3'=0$.
        \begin{itemize}
        \item If $\beta_{3}=0,$ then we get the affgebra
$L_1(\beta_{1},0,\beta_{1},0,1,\gamma_{3},N_1,0,0).$
\item If $\beta_{3}\neq 0,$ we get the affgebra
$L_{10}(\beta_{1},1,\beta_{1},0,1,\gamma_{3},N_1,0,0).$
\end{itemize}

\item Let $\gamma_{3} \neq 1$ and $(\beta_{1}-\gamma_{3})(N_1-\beta_{1}(1-\beta_{1}))= 0,$ then taking  $C_3=\frac{\alpha_{3}N_3+\alpha_{2} (N_1-\beta_{1}(\gamma_{3}-\beta_{1})}{\alpha_{3}  (\gamma_{3}-1)},$ we obtain $N_3'=0$. Thus,
we have  $\beta_{3}'=\alpha_{3} \beta_{3},$ and $N_2'=\alpha_{3} N_2.$

\begin{itemize}
        \item If $\beta_{3}=0,$ then we get the affgebras
$L_1(\beta_{1},0,\beta_{1},0,1,\gamma_{3},N_1,0,0)$
and $L_2(\beta_{1},0,\beta_{1},0,1,\gamma_{3},N_1,1,0)$
depending on whether $N_2=0$ or not.
\item If $\beta_{3}\neq 0,$ then we can suppose  $\beta_{3}'=1,$ and obtain the affgebras
$L_{11}(\beta_{1},1,\beta_{1},0,1,\beta_{1},N_1,N_2,0)$ and
$L_{12}(\beta_{1},1,\beta_{1},0,1,\gamma_{3},\beta_{1}(1-\beta_{1}),N_2,0).$
    \end{itemize}

 \item Let $\gamma_{3} = 1$ and $N_1\neq \beta_{1}(1-\beta_{1}),$  then taking  $\alpha_{2}=\frac{\alpha_{3} N_3}{\beta_{1}(1-\beta_{1}) - N_1},$ we have $N_3'=0$ and obtain $\beta_{3}'=\alpha_{3}\beta_{3},$ $N_2'=\alpha_{3} N_2.$

\begin{itemize}
        \item If $\beta_{3}=0,$ then we get the affgebras
$L_1(\beta_{1},0,\beta_{1},0,1,1,N_1,0,0)$
and $L_2(\beta_{1},0,\beta_{1},0,1,1,N_1,1,0)$
depending on whether $N_2=0$ or not.
\item If $\beta_{3}\neq 0,$ then we can suppose  $\beta_{3}'=1,$ and obtain the affgebra
$L_{13}(\beta_{1},1,\beta_{1},0,1,1,N_1,N_2,0).$
    \end{itemize}
 \item Let $\gamma_{3} = 1$ and $N_1= \beta_{1}(1-\beta_{1}),$ then we get $\beta_{3}'=\alpha_{3}\beta_{3},$ $N_2'=\alpha_{3} N_2$ and
 $N_3'=\alpha_{3} N_3.$
 \begin{itemize}
        \item If $N_3=0,$ then we get the affgebras
$L_1,$ $L_2$ and $L_{13},$ with $\beta_3=\beta_1,$ $\gamma_3=1,$ $N_1 = \beta_1(1-\beta_1).$
\item If $N_3\neq0,$ then we may assume $N_3'=1$ and obtain the affgebra
$L_{14}(\beta_{1},\beta_{3},\beta_{1},0,1,1,\beta_{1}(1-\beta_{1}),N_2,1).$
    \end{itemize}

  \end{itemize}

\item  Let $\beta_{1} = \beta_{5}$ and $\gamma_{2} = 0,$ then $\gamma_{2}'=0.$ Thus, we get that $$\beta_{3}'=\alpha_{4}\beta_{3}, \quad \gamma_{1}'=\alpha_{2} (\beta_{1}-\gamma_{3})+\alpha_{4} \gamma_{1},$$ $$N_2'=\alpha_{3} N_2+\alpha_{1} (N_1 - \beta_{1} (1-\beta_{1})), \quad N_3'=\alpha_{2} N_1 + \alpha_{4}N_3+ \alpha_{4} C_3 (1-\gamma_{3})$$

\begin{itemize}
        \item Let $N_1 \neq \beta_{1} (1-\beta_{1})$ and $\gamma_{3} \neq 1,$ then taking
            $\alpha_{1}=\frac{\alpha_{3} N_2}{\beta_{1} (1-\beta_{1})-N_1},$ $C_3=\frac{\alpha_{2} N_1 + \alpha_{4}N_3}{\alpha_{4}(\gamma_{3}-1)},$ we obtain $N_2'=0,$ $N_3'=0.$
            \begin{itemize}
        \item If $\gamma_{3} \neq \beta_{1},$ then taking $\alpha_{2}=\frac{\alpha_{4} \gamma_{1}}{\gamma_{3}-\beta_{1}},$ we get $\gamma_{1}'=0$ and obtain the  affgebras
             $L_4(\beta_{1},0,\beta_{1},0,0,\gamma_{3},N_1,0,0)$ and
              $L_{15}(\beta_{1},1,\beta_{1},0,0,\gamma_{3},N_1,0,0),$ depending on whether $\beta_{3} = 0$ or not.
              \item If $\gamma_{3} = \beta_{1},$ then in case of $\gamma_{1}=0,$ we obtain the  affgebras $L_4$ and $L_{15}$ with $\gamma_3=\beta_3=\beta_1.$
              In the case of $\gamma_{1}\neq 0,$ we obtain the affgebra $L_{16}(\beta_{1},\beta_{3},\beta_{1},1,0,\beta_{1},N_1,0,0).$
     \end{itemize}
     \item Let $N_1 \neq \beta_{1} (1-\beta_{1})$ and $\gamma_{3} = 1,$ then taking
            $\alpha_{1}=\frac{\alpha_{3} N_2}{\beta_{1} (1-\beta_{1})-N_1},$ we obtain $N_2'=0.$
             \begin{itemize}
        \item If $N_1 \neq 0,$ then taking $\alpha_{2} =- \frac{\alpha_{4}N_3}{N_1},$ we have $N_3'=0$ and $\beta_{3}'=\alpha_{4}\beta_{3}, $ $\gamma_{1}'=\alpha_{4} \gamma_{1}.$
        \begin{itemize}
        \item If $\gamma_{1}=0,$ then we get the affgebras $L_4$ and $L_{15}$ with $\gamma_3=1,$ $\beta_3=\beta_1.$
        \item If $\gamma_{1}\neq 0,$ then we may assume $\gamma_{1}'=1,$
        and obtain the affgebra $L_{17}(\beta_{1},\beta_{3},\beta_{1},1,0,1,N_1,0,0).$
\end{itemize}
\item If $N_1 = 0,$ then $\beta_{1} (1-\beta_{1})\neq 0.$ Taking $\alpha_{2}  = \frac{\alpha_{4} \gamma_{1}}{1-\beta_{1}},$ we have $\gamma_{1}'=0$ and
$\beta_{3}'=\alpha_{4}\beta_{3},$ $N_3'=\alpha_{4}N_3.$
\begin{itemize}
        \item If $N_3=0,$ then we get the affgebras $L_4$ and $L_{15}$ with $\beta_3=\beta_1,$ $N_1=0$.
        \item  $N_3\neq 0,$ then we may assume $N_3'=1,$ and obtain the affgebra $L_{18}(\beta_{1},\beta_{3},\beta_{1},0,0,1,0,0,1).$

\end{itemize}
\end{itemize}
\item Let $N_1 = \beta_{1} (1-\beta_{1})$ and $\gamma_{3} \neq 1,$ then taking
            $C_3=\frac{\alpha_{2} N_1 + \alpha_{4}N_3}{\alpha_{4}(\gamma_{3}-1)},$ we obtain  $N_3'=0.$
            \begin{itemize}
        \item If $\gamma_{3} \neq \beta_{1},$ then taking $\alpha_{2}=\frac{\alpha_{4} \gamma_{1}}{\gamma_{3}-\beta_{1}},$ we may suppose $\gamma_{1}'=0.$
        \begin{itemize}
        \item If $\beta_{3}=0,$ $N_2 =0,$ then we obtain the affgebra
$L_4$ with $N_1=\beta_{1} (1-\beta_{1}).$
\item If $\beta_{3}=0,$ $N_2 \neq 0,$ then we obtain the affgebra
$L_7$ with $N_1=\beta_{1} (1-\beta_{1}).$
\item If $\beta_{3}\neq 0,$ $N_2 =0,$ then we obtain the affgebra
$L_{15}$ with $N_1=\beta_{1} (1-\beta_{1}).$

\item If $\beta_{3}\neq 0,$ $N_2 \neq 0,$ then we obtain the affgebra
$L_{19}(\beta_{1},1,\beta_{1},0,0,\gamma_{3},\beta_{1} (1-\beta_{1}),1,0).$
        \end{itemize}
 \item If $\gamma_{3}= \beta_{1},$ then we have
$\beta_{3}'=\alpha_{4}\beta_{3},$ $\gamma_{1}'=\alpha_{4} \gamma_{1},$ $N_2'=\alpha_{3} N_2.$

\begin{itemize}
        \item If $\gamma_{1}=0,$ then similarly to the previous case we obtain the affgebras $L_4,$ $L_{7},$ $L_{15}$ and $L_{19}$ with $\gamma_3=\beta_3=\beta_1$ and $N_1 = \beta_1(1-\beta_1).$

        \item If $\gamma_{1}\neq 0,$ then we can suppose $\gamma_{1}'=1,$ and obtain the affgebras $L_{16}(\beta_{1},\beta_{3},\beta_{1},1,0,\beta_{1},\beta_{1} (1-\beta_{1}),0,0)$ and $L_{20}(\beta_{1},\beta_{3},\beta_{1},1,0,\beta_{1},\beta_{1} (1-\beta_{1}),1,0)$ depending on whether $N_2=0$ or not.
   \end{itemize}
   \end{itemize}
  \item Let $N_1 = \beta_{1} (1-\beta_{1})$ and $\gamma_{3} = 1.$
  \begin{itemize}
        \item Let $\beta_{1}\neq 1,$ then taking $\alpha_{2} = \frac{\alpha_{4}\gamma_{1}}{1-\beta_{1}},$ we get that $\gamma_{1}'=0.$
\begin{itemize}
        \item If $N_3=0,$ then similarly to the previous case, we get the affgebras $L_4,$ $L_{7},$ $L_{15}$ and $L_{19}$ with $\gamma_3=1,$ $\beta_3=\beta_1$ and $N_1 = \beta_1(1-\beta_1).$

        \item If $N_3\neq0,$ then we may suppose $N_3'=1,$ and obtain the algebras
        $L_{21}(\beta_{1},\beta_{3},\beta_{1},0,0,1,\beta_{1} (1-\beta_{1}),0,1)$ and $L_{22}(\beta_{1},\beta_{3},\beta_{1},0,0,1,\beta_{1} (1-\beta_{1}),1,1)$ depending on whether $N_2=0$ or not.

        \end{itemize}
    \item Let $\beta_{1}= 1.$
\begin{itemize}
        \item If $\gamma_{1}=0,$ then we get the affgebras
        $L_4,$ $L_{7},$ $L_{15},$ $L_{19},$ $L_{21}$ and $L_{22},$ with $\gamma_3=\beta_3=\beta_1=1$ and $N_1 = 0.$
          \item If $\gamma_{1}\neq0,$ then we may suppose $\gamma_{1}'=1,$ and obtain the affgebras
        $L_{23}(1,\beta_{3},1,1,0,1,0,0,N_3)$ and $L_{24}(1,\beta_{3},1,1,0,1,0,1,N_3)$ depending on whether $N_2=0$ or not.

        \end{itemize}
        \end{itemize}
  \end{itemize}

\end{itemize}

\end{proof}

{\small
\begin{tabular}{p{9cm}}

\

  Berdalova Kh.R.,\\
    Institute of Mathematics, Uzbekistan Academy of Sciences\\
    email: xursanoyberdalova@gmail.com\\[2mm]
      Khudoyberdiyev A.Kh.,\\
   Institute of Mathematics, Uzbekistan Academy of Sciences, National University of Uzbekistan, Tashkent, Uzbekistan\\
    email: khabror@mail.ru\\

\end{tabular}
}

\end{document}